\documentclass[12pt]{article}
\usepackage[margin=2cm]{geometry}
\usepackage{amsmath,amstext,amsfonts,amsthm,amssymb,bbm,hyperref}

\def\mes{\operatorname{mes}}

\def\Sym{\operatorname{Sym}}

\def\supp{\operatorname{supp}}

\theoremstyle{plain}
\newtheorem{thm}{Theorem}
\newtheorem{defin}{Definition}[section]
\newtheorem{qn}[defin]{Problem}

\newtheorem{rmk}[defin]{Remark}
\newtheorem{cor}[defin]{Corollary}
\newtheorem{lemma}[defin]{Lemma}
\newtheorem*{thm*}{Theorem}

\begin{document}

\author{Ilya Goldsheid\textsuperscript{1} and Sasha Sodin\textsuperscript{1,2}}
\footnotetext[1]{School of Mathematical Sciences,
Queen Mary University of London, 
Mile End Road, London E1 4NS, UK. email: [i.goldsheid, a.sodin]@qmul.ac.uk.}
\footnotetext[2]{Supported in part by a Royal Society Wolfson Research Merit Award (WM170012), a Philip Leverhulme Prize of the Leverhulme Trust (PLP-2020-064), and a Morris Belkin Visiting Professorship at the Weizmann Institute of Science.}

\title{Sets of  non-Lyapunov behaviour for scalar and matrix Schr\"odinger  cocycles}
\maketitle

\begin{abstract}
We discuss  the growth of the singular values of symplectic transfer matrices associated with ergodic discrete Schr\"odinger operators in one dimension, with scalar and matrix-valued potentials. While for an individual value of the spectral parameter the rate of exponential growth is almost surely governed by the Lyapunov exponents, this is not, in general, true simultaneously for all the values of the parameter. 
The structure of the exceptional sets is interesting in its own right, and is also of importance in the spectral analysis of the operators. We present  new results along with amplifications and generalisations of several  older ones, and also list a few open questions.

Here are two sample results. On the negative side,  for any square-summable sequence $p_n$ there is a residual set of energies in the  spectrum on which the middle singular value (the $W$-th out of $2W$) grows no faster than $p_n^{-1}$. On the positive side, for a large class of cocycles including the i.i.d.\ ones,  the set of energies at which the growth of the singular values is not as given by the Lyapunov exponents has zero Hausdorff measure with respect to any gauge function $\rho(t)$ such that $\rho(t)/t$ is integrable at zero. 

The employed arguments from the theory of subharmonic functions also yield a generalisation of the Thouless formula, possibly of independent interest: for each $k$, the average of the first $k$ Lyapunov exponents is the logarithmic potential of a probability measure.
\end{abstract}

\section{Introduction}

Let $(\Omega, \mathcal B, \mathbb P)$ be a probability space, and let $\mathbf T: \Omega \to \Omega$ be an ergodic transformation. Also let $\mathbf F : \Omega \to \Sym_W(\mathbb R)$ be a measurable function from $\Omega$ to the space of symmetric $W \times W$ matrices such that
\begin{equation}\label{eq:logint}
\mathbb E \log_+ \| \mathbf F(\omega) \| < \infty~.
\end{equation}
We refer to $\mathbf F$ as the potential function.

To every $\omega \in \Omega$ we associate a sequence of $W \times W$ matrices $V_{n, \omega}  = \mathbf  F(\mathbf T^n \omega)$, $n \geq 0$. Then for $z \in \mathbb C$ consider the $2W\times 2W$ symplectic matrices
\[ T_{n, \omega} (z) = \left( \begin{array}{cc} z \mathbbm 1 - V_{n,\omega} &  - \mathbbm 1 \\ \mathbbm 1 & 0 \end{array} \right)~, \quad
\Phi_{n,\omega}(z)  =  
T_{n, \omega} (z) T_{n-1,\omega}(z)  \cdots T_{1, \omega} (z)~.
\]
The matrices $T_{n,\omega}$ and $\Phi_{n,\omega}$ appear naturally as the transfer matrices of the (matrix) discrete Schr\"odinger operator $H_\omega$ acting on a dense subspace of $\mathbb \ell_2(\mathbb N \to \mathbb R^W)$ via
\[ (H_\omega \psi)(n) = \begin{cases}
\psi(n+1) + V_{n,\omega} \psi(n) + \psi(n-1)~, &n \geq 2~; \\
\psi(2) + V_{1,\omega} \psi(1)~, &n=1~. 
\end{cases}\]
Due to the cocycle property $\Phi_{n+m,\omega}(z) = \Phi_{n,T^m \omega}(z)   \Phi_{m,\omega}(z)$, the matrices $\Phi_{n,\omega}(z)$ are said to form a matrix Schr\"odinger cocycle.

For $j =1, \cdots, 2W$, let 
\[ \gamma_j(z) = \lim_{n \to \infty} \frac1n \mathbb E \log s_j(\Phi_{n,\omega}(z))  \]
be the $j$-th Lyapunov exponent (where $s_j$ stands for the $j$-th largest singular value). The assumption (\ref{eq:logint}) implies that 
these numbers are finite. Due to the symplectic structure, $s_{2W+1-j}(\Phi_{n,\omega}(z)) = s_j^{-1}(\Phi_{n,\omega}(z))$, whence $\gamma_{2W+1-j}(z) = - \gamma_j(z)$.
According to a theorem of Furstenberg and Kesten \cite{FK},  for any fixed $z \in \mathbb C$ one has  almost surely
\begin{equation}\label{eq:lyap}
\lim_{n \to \infty} \frac1n \mathbb  \log \, s_j(\Phi_{n,\omega}(z)) =  \gamma_j(z)~.
\end{equation}
However, the Furstenberg--Kesten theorem does not imply that (almost surely) the Lyapunov behaviour (\ref{eq:lyap}) holds simultaneously for all $z \in \mathbb C$, and in fact this is not always true. Most of the proofs of Anderson localisation (for various classes of  quasi-one-dimensional operators such as $H_\omega$ and their full-line counterparts) require understanding the structure of the sets of non-Lyapunov behaviour, or at least demonstrating that these sets do not support spectral measure. Thus, numerous  properties of these sets have been established for specific operators and dynamical systems, especially, in the scalar case $W = 1$.

On the other hand,  relatively few general results (not restricted to special dynamical systems) are available in the published literature, even in the scalar case; the matrix case is virtually unexplored. The goal of the current note is to summarise what is known to us at the moment,  avoiding unnecessarily restrictive  assumptions, as much as possible. We also collect several open questions which, in our opinion,  deserve further investigation.

\paragraph{Assumptions}
For convenience of reference, we collect our assumptions in the following display:
\begin{equation}\label{eq:assum}
\text{$\mathbf T$ is ergodic, $\mathbf F: \Omega \to \Sym_W(\mathbb R)$ is measurable and $\mathbb E \log_+ \|\mathbf F\| < \infty$.}
\end{equation}
Some results can be strengthened in the uniquely ergodic case:
\begin{equation}\label{eq:assum-ue}
\left\{\begin{aligned}
&\text{$\Omega$ is a compact, metrisable topological space,}\\
&\text{$\mathbf T$ is continuous and uniquely ergodic,}\\
&\text{$\mathbf F: \Omega \to \Sym_W(\mathbb R)$ is continuous:}
\end{aligned}\right.
\end{equation}
in this setting, some of the results hold for all (rather than just almost all) values of $\omega \in \Omega$.
We shall write that a property holds almost* surely if it holds almost surely under the assumptions (\ref{eq:assum}) and for all $\omega$ under the assumptions (\ref{eq:assum-ue}).

On the opposite extreme from unique ergodicity, we single out   the   special case of independent matrices:
\begin{equation}\label{eq:assumiid}
\left\{\begin{aligned}
&\text{$V_n$ are independent, identically distributed, with $\mathbb E \|V_0\|^\eta < \infty$ for some $\eta > 0$,}\\
&\text{and the support of $V_0$ is irreducible and contains $V, V'$  such that $\operatorname{rk} (V-V') = 1$.}
\end{aligned}\right.
\end{equation}
The r\^ole of the condition on the support is as follows: in this case, the combination of the result of \cite{G95} with \cite{GM} implies that 
the Lyapunov spectrum is simple for all $z \in \mathbb C$:
\[ \gamma_1(z) > \cdots > \gamma_W(z) > 0~. \]
Furthermore, one has the locally uniform large deviation estimate, going back to the work of Le Page \cite{LeP} (see e.g.\ \cite{DK})
\begin{equation}\label{eq:largedev}
\forall R, \epsilon > 0 \,\, \exists C, c > 0: \,\, \forall |E| < R, \,  1 \leq j \leq W \,\, \mathbb P \left\{ \left| \frac{1}{n} \log s_j(\Phi_{n,\omega}(E)) - \gamma_j(E) \right| \geq \epsilon \right\} \leq C e^{-c n}~.
\end{equation}
Estimates similar to  (\ref{eq:largedev}) are also available for other dynamical systems, for example, for irrational shifts with Diophantine frequencies when the potential function is analytic (see, particularly, the work of Goldstein and Schlag \cite{GS}, the survey of Schlag \cite{Schlag}, and references therein).

\paragraph{Sub-Lyapunov behaviour of $\gamma_W$} 
We first consider the special case $j = W$ (which boasts the closest relation with the spectral properties of the operator $H_\omega$). In this case (\ref{eq:lyap}) almost surely holds (simultaneously) for all $z$ outside the essential spectrum $S$ of $H_\omega$ (it is known that $S$ is almost surely equal to a deterministic set). This was proved for $W=1$ in \cite[Theorem 8]{G75} and (in the continual setting) by Johnson in \cite{J}; see also the recent work of Zhang \cite{Zh}. For larger $W$, the arguments in the cited works yield the following (cf.\  Johnson--Nerurkar \cite{JN}):
\begin{thm}[after \cite{G75,J,JN}]\label{thm:outsidesp}
Assume (\ref{eq:assum}). Then almost* surely one has the following: for any compact $K \subset \mathbb C \setminus S$,
\begin{eqnarray}
\label{eq:unifhyp1}
&&\frac1n \sum_{j=1}^W \log \, s_j(\Phi_{n,\omega}(z)) \rightrightarrows  \sum_{j=1}^W \gamma_j(z) \quad \text{on $K$}~,\\
\label{eq:unifhyp} 
&&\liminf_{n \to \infty} \inf_{z \in K} \left[ \frac1n   \log \, s_W(\Phi_{n,\omega}(z))  - \gamma_W(z) \right] \geq 0~.\end{eqnarray}
\end{thm}
\noindent Here (\ref{eq:unifhyp}) follows from (\ref{eq:unifhyp1}) and a general result stated below as (\ref{eq:unif}), applied with $k = W-1$. We do not know whether necessarily
\[  \frac1n   \log \, s_W(\Phi_{n,\omega}(z))  \to \gamma_W(z) \] 
on $\mathbb C \setminus S$ (or even on $\mathbb C \setminus \mathbb R$); see further below.

\medskip \noindent
On the essential spectrum $S$, the behaviour of the singular values is completely different. For $W=1$, it was shown in \cite{G75,G80a} that  the set 
\[ \Lambda_{\tau,\omega}^W = \left\{ E \in \mathbb R\, :  \, \liminf_{n \to \infty} \frac{1}{n} \log s_W(\Phi_{n,\omega}(E)) \leq \tau\gamma_W(E)\right\} \]
is dense in  $S$ for any $1 \geq \tau \geq 1/2$.
Further works on the subject include Carmona \cite{Car}, Avron--Simon \cite{AS}, and del Rio--Makarov--Simon \cite{dRMS}. Recently, Gorodetski and Kleptsyn \cite{GorKl} studied the  question for products of (bounded) independent random matrices, allowing for more general (non-Schr\"odinger) dependence on $z$, and showed that in this case $\Lambda_{0,\omega}^{W=1}$ is almost surely residual in $S$ (i.e.\ its complement in $S$ is a countable union of nowhere dense sets). We are not aware of any published works pertaining to $W > 1$, with the exception of a short discussion in \cite{GS}.

The following strengthening and generalisation of \cite{G75,G80a} and most of the subsequent results was proved by the first author in the early 1990-s. We provide a proof in Section~\ref{s:liminf}.
\begin{thm}\label{thm:liminf} Assume (\ref{eq:assum}). Then: (a) $\Lambda_{0,\omega}^W \subset S$ is almost* surely residual (and in particular dense) in $S$. (b) If the potential function $\mathbf F$ is bounded, then for any non-negative square-summable sequence $(p_n)_{n \geq 1}$   the following set is  also almost* surely residual in $S$:
\begin{equation}\label{eq:nonlyap-liminf} 
 \left\{ E \in S \, : \, \liminf_{n \to \infty} p_n s_W(\Phi_{n, \omega}(E)) = 0 \right\}~.\end{equation}
\end{thm}

\begin{qn} Let $W > 1$, $1 \leq j \leq W-1$. Does (\ref{eq:lyap}) hold  (almost surely) for all $z \in \mathbb C$? If not, where are the points at which (\ref{eq:lyap})  can fail located? What is the closure of the set 
\[ \left\{ E \in \mathbb R\, :  \,  \liminf_{n \to \infty} \frac{1}{n} \log s_j(\Phi_{n,\omega}(E)) < \gamma_j(E)\right\}? \]
\end{qn}

We believe (with some support from numerical simulations) that   (\ref{eq:lyap}) is sometimes violated also for $j < W$. However, if this is not the case, this would also have interesting spectral-theoretic applications -- see \cite{GS}. 

\paragraph{Subharmonic functions}

What follows is based on considerations from the  theory of subharmonic functions, developing ideas applied in the current context by Craig--Simon \cite{CS} and further by Simon \cite{Simon}. It is clear that the functions
\begin{equation}
z \mapsto \frac{1}{n}  \log \| \Phi_{n,\omega}(z) \|~, \quad
z \mapsto \frac{1}{n} \mathbb E \log \| \Phi_{n,\omega}(z) \|
\end{equation}
are subharmonic. Considering the wedge powers of $\Phi_{n,\omega}(z)$, we obtain that so are 
\begin{eqnarray}\label{eq:subh1}
&&z \mapsto \frac{1}{n}  \sum_{j=1}^k \log s_j( \Phi_{n,\omega}(z) )~, \\
\label{eq:subh2}
&&z \mapsto \frac{1}{n} \mathbb E \sum_{j=1}^k \log s_j( \Phi_{n,\omega}(z) )~,
\end{eqnarray}
for any $1 \leq k \leq W$. Craig and Simon showed \cite[Theorem 2.1]{CS} (for $W=1$) that $\gamma_1(E)$ is upper semicontinuous and thus subharmonic. Their argument (based on the observation that (\ref{eq:subh2}) is non-increasing on the sequence $n = 1, 2,4,8, \cdots$) applies equally well to larger $W$ and to the wedge powers of $\Phi_n(E)$, thus  the functions 
\begin{equation}\label{eq:Gamma} \Gamma_k(z) = \gamma_1(z) + \cdots + \gamma_k(z) \end{equation}
are also subharmonic.  

By definition,  (\ref{eq:subh2}) converge pointwise to $\Gamma_k(z)$, whereas the Furstenberg--Kesten theorem \cite{FK} implies that (\ref{eq:subh1})  converge to $\Gamma_k(z)$ almost everywhere. However, these notions  of  convergence are not the most convenient ones for studying sequences of subharmonic functions, and it turns out to be useful to work with the topology of distributional convergence (cf.\ Azarin \cite{Az1, Az} where this approach is systematically applied to the study of asymptotic behaviour of subharmonic functions).  Denote by $\mathcal D$ the  space of infinitely differentiable compactly supported functions from $\mathbb C$ to $\mathbb R$, and by $\mathcal D'$ the dual space of distributions.
\begin{rmk}Following \cite[Theorems 3.2.12, 3.2.13]{Hor}, we note that convergence of a sequence of subharmonic functions to a limit in $\mathcal D'$ is equivalent to convergence in $L_p^{\operatorname{loc}}$ for any $1 \leq p < \infty$, and implies the convergence of derivatives in $L_p^{\operatorname{loc}}$ for any $1 \leq p < 2$. See further Remark~\ref{rem:AzF}.
\end{rmk}
\noindent The following result, playing an important r\^ole in the sequel, is proved in   Section~\ref{s:cs}.
\begin{thm}[Motivated by Craig--Simon] \label{thm:cs} Assume (\ref{eq:assum}).  Then: (a) The functions (\ref{eq:subh2}) converge to $\Gamma_k$   in $\mathcal D'$, and the same is almost* surely true for the functions (\ref{eq:subh1}). (b) There exist probability measures $\mu_1, \cdots, \mu_W$ on $\mathbb C$ such that 
\begin{equation}\label{eq:th}\int \log_+ |z| d\mu_k(z) < \infty~, \,\, \sup_{z \in \mathbb C, 0 < r \leq 1/e} \frac{\mu_k(B(z, r)) |\log r|}{1 + \log_+ |z|} < \infty~, \,\,   \frac{\Gamma_k(z)}{k}  = \int \log |z-z'| d\mu_k(z)~.\end{equation}
\end{thm}

\begin{rmk}
Part (b) of Theorem~\ref{thm:cs} can be seen as a generalisation of the Thouless formula. The version of the latter for the strip, as proved by Craig--Simon \cite{CS2}, asserts that the mean Lyapunov exponent $\Gamma_W/W$ is the logarithmic potential of the integrated density of states of $H$; accordingly, $\supp \mu_W = S \subset \mathbb R$. The cases $1 \leq k < W$ are new.
\end{rmk}
\begin{qn}
What is the support $\supp \mu_k$ of $\mu_k$ for  $1 \leq k < W$?
\end{qn}

\paragraph{Convergence outside small sets}

While the exceptional sets in Theorem~\ref{thm:liminf} are large in topological sense, metrically, they are very small. We focus on the slowest (smallest positive) Lyapunov exponent $\gamma_W$ and on cocycles for which the large deviation estimates (\ref{eq:largedev}) are available (e.g.\ the independent case (\ref{eq:assumiid})).  The following result is inspired by the work of Gorodetski--Kleptsyn \cite{GorKl}.

Let $\mes_\rho$ be the Hausdorff measure defined by a gauge function $\rho(t)$ (a continuous increasing function vanishing at zero) . 
\begin{thm}\label{thm:rho} 
Suppose the cocycle $\Phi_{n,\omega}$ satisfies the large deviation estimate (\ref{eq:largedev}).
Then  almost* surely 
\begin{equation}\label{eq:mesrho0'} \mes_\rho \left\{ z \in \mathbb C \, : \,  \liminf_{n \to \infty} \frac{1}{n} \log s_W(\Phi_{n,\omega}(z)) < \gamma_W(z) \right\} = 0  \end{equation}
for  any gauge function $\rho$ such that  
\begin{equation}\label{eq:condrho}\int_0^1 \frac{\rho(t) dt}{t} < \infty~.\end{equation}
\end{thm}
In the independent and bounded (but not necessarily Schr\"odinger) case, $W=1$, Gorodetski and Kleptsyn \cite{GorKl} showed that (\ref{eq:mesrho0'}) holds for the functions $\rho(t) = t^\epsilon$, $\epsilon > 0$. A similar question for   Schr\"odinger cocycles ($W=1$) corresponding to general ergodic systems was considered by Simon \cite{Simon}.\footnote{In the proof of \cite[Theorem 1.16]{Simon} it is asserted that one has   convergence (\ref{eq:lyap}) quasi-everywhere, i.e.\ outside a set of zero logarithmic capacity, in the general setting (\ref{eq:assum}) with $W=1$; however, the  argument there only yields the counterpart of Corollary~\ref{thm:simon-corr} below. As we discuss in Corollary~\ref{cor:sim+} below, this is still sufficient to establish the conclusion of \cite[Theorem 1.16]{Simon}, thus the latter is valid as stated. Kleptsyn and Quintino \cite{KQ, Q} conjecture that  quasi-everywhere convergence in (\ref{eq:lyap}) fails in the independent case (\ref{eq:assumiid}).} 
\begin{rmk}
In a companion work \cite{Jar}, we show that the condition (\ref{eq:condrho}) can not in general be improved, as the following holds in the independent case (\ref{eq:assumiid}) with $W=1$ and sufficiently regular distribution of the potential. If $\rho$ is a gauge function such that $\rho(t)/t$ is non-increasing on $(0, 1)$ and non-integrable, then for any $\tau > 0$ and any interval $I \subset S$
\[ \mes_\rho \left\{ E \in I \, : \,  \liminf_{n\to \infty} \frac{1}{n} \log \|\Phi_{n,\omega}(E) \| \leq \tau \gamma_1(z) \right\} = \infty~. \]
\end{rmk}

\begin{rmk}\label{rem:AzF}
We do not know whether the conclusion of Theorem~\ref{thm:rho} holds in the general case (\ref{eq:assum}). However, it is proved in \cite{Az,Az1,Fav} that convergence of subharmonic functions $u_n \to u$ in $\mathcal D'$ implies the convergence in $\rho$-Hausdorff content for any $\rho$ satisfying (\ref{eq:condrho}), i.e.\
\begin{equation}\label{eq:incontent} \forall \epsilon>0 \,\, \operatorname{cont}_\rho \left\{ z \in \mathbb C \, : \, |u_n(z)-u(z)|>\epsilon \right\} \underset{n \to \infty}\longrightarrow 0~.\end{equation}
In particular, almost* surely $n^{-1} \log s_j(\Phi_{n,\omega}) \to \gamma_j$ in the topology (\ref{eq:incontent}).
\end{rmk}

\medskip \noindent 

\paragraph{No super-Lyapunov behaviour}

Theorem~\ref{thm:cs}-(a) combined with general properties of subharmonic functions \cite[Theorem 3.2.13]{Hor} implies the following (locally uniform) version of \cite[Theorem 2.3]{CS}, extended to arbitrary $W$.

\begin{cor}[after Craig--Simon]\label{thm:cs-unif} Let $1 \leq k \leq W$. Then almost* surely 
\begin{equation}\label{eq:upper} \forall z \in \mathbb C\, : \,  \limsup_{n \to \infty} \left[  \sum_{j=1}^k \frac1n \log s_j(\Phi_{n, \omega}(z))  - \Gamma_k(z) \right] \leq 0~.\end{equation}
Moreover, if $\Gamma_k$ is continuous on a compact set $K \subset \mathbb C$, then almost* surely
\begin{equation}\label{eq:unif}\limsup_{n \to \infty} \sup_{z \in K} \left[  \sum_{j=1}^k \frac1n \log s_j(\Phi_{n, \omega}(z))  - \Gamma_k(z) \right] \leq 0~.\end{equation}
\end{cor}

\begin{qn} Does a similar bound hold for the individual singular values $s_j(\Phi_{n,\omega})$, $1 < j \leq W$?
\end{qn}

\begin{rmk} The relation (\ref{eq:unif}) clearly fails if $\Gamma_k$ is not continuous on $K$. See Damanik--Gan--Kr\"uger \cite{DGK} and references therein for examples where the Lyapunov exponent is discontinuous (for $W = 1$).
\end{rmk}

\paragraph{Subsequential Lyapunov behaviour} In the case of independent matrices (\ref{eq:assumiid}) with $W=1$ (and also in a more general setting, still assuming independence but allowing for more general dependence of the matrices on $z$), Gorodetski--Kleptsyn showed \cite{GorKl} that (\ref{eq:upper}) is in fact a pointwise equality: almost surely
\[ \forall E  \in  \mathbb R \, : \,  \limsup_{n \to \infty}    \frac1n \log s_1(\Phi_{n, \omega}(E))  =  \gamma_1(E) \]
(the non-real values of $z$ are in this case covered by Theorem~\ref{thm:outsidesp}).
The following theorem (proved in \cite{GS} although stated there in a less general setting), extends this result to arbitrary $W$ with an explicit description of the subsequences on which the convergence has to hold.
\begin{thm}[after \cite{GorKl,GS}]\label{thm:subseq}  Assume the large deviation estimate (\ref{eq:largedev}). For any sequence $N_n \to \infty$ such that $N_n / n \to \infty$ and $(\log N_n) / n \to 0$ one has  almost surely
\[ \forall E \in \mathbb R \, \: \, 
\lim_{n \to \infty} \min( \max_{1 \leq j \leq W}   | \frac1n \log s_j(\Phi_{n, \omega}(E))   -   \gamma_j(E)|, \max_{1 \leq j \leq W}   | \frac1{N_n} \log s_j(\Phi_{N_n, \omega}(E))   -   \gamma_j(E)|) = 0~.
\]
In particular, almost surely 
\begin{equation}\label{eq:limsup} \forall E \in \mathbb R\, : \,  \liminf_{n \to \infty}   \max_{1 \leq j \leq W}  | \frac1n \log s_j(\Phi_{n, \omega}(E))   -   \gamma_j(E)| = 0~. \end{equation}
\end{thm}
\begin{rmk}
If instead of (\ref{eq:largedev}) one assumes a large deviation estimate valid throughout the complex plane, then also the conclusion of Theorem~\ref{thm:subseq} will hold for all the complex (rather than just real) values of the parameter. 
\end{rmk}
\begin{qn} Does (\ref{eq:limsup}) hold in the general case (\ref{eq:assum})?
\end{qn}
From Theorem~\ref{thm:cs}-(a) and general properties of subharmonic functions (see \cite[Theorem 2.7.4.1.]{Az}) we have the following weaker property (proved, for $W = 1$, in \cite{Simon}):
\begin{cor}[after Simon]\label{thm:simon-corr} Assume (\ref{eq:assum}). Then  almost* surely  
\[ \forall 1 \leq k \leq W \quad \liminf_{n \to \infty}    | \frac1n \sum_{j=1}^k \log s_j(\Phi_{n, \omega}(z))   -   \Gamma_k(z)| = 0 \]
for all $z \in \mathbb C$ outside a set of zero logarithmic capacity.
\end{cor}

Note that a set of zero logarithmic capacity has zero $\rho$-Hausdorff measure for any $\rho$ satisfying (\ref{eq:condrho}) (see  Frostman \cite[\S 47]{Fr}). The opposite implication is not,  in general, true (see Carleson \cite{Carl}); however, Erd\H{o}s and Gillis \cite{EG} showed that any set with finite $(\log^{-1} (e/t))$-Hausdorff measure is of zero logarithmic capacity (the converse is also not true, see again \cite{Carl}).

\paragraph{Spectral-theoretic applications}
In the case $k=W$, Corollary~\ref{thm:simon-corr} can be augmented as follows (see Section~\ref{s:limsup} for the proof). Recall that a subspace $F \subset \mathbb R^{2W}$ (of dimension $W$) is called Lagrangian if 
\[ F = (JF)^\perp~, \quad \text{where} \quad J = \left( \begin{array}{cc} 0 & - \mathbbm 1 \\ \mathbbm 1 & 0 \end{array} \right)~. \]
\begin{thm}\label{thm:limsup} Assume (\ref{eq:assum}); fix a Lagrangian subspace $F \subset \mathbb R^{2W}$ and let $\pi_F: \mathbb R^{2W} \to F$ be the orthogonal projection. Then almost* surely 
\begin{equation}\label{eq:limsup1} \limsup_{n \to \infty}  \frac1n  \log s_W(\Phi_{n, \omega}(z) \pi_F^*)   \geq   \gamma_W(z)  \end{equation}
for all $z \in \mathbb C$ outside a set of zero logarithmic capacity.
\end{thm}

\noindent Theorem~\ref{thm:limsup} implies the following result, established by Simon \cite{Simon} in the special case $W=1$ (see further Poltoratski--Remling \cite[Theorem 4.2]{PR}  for a generalisation to the  non-ergodic setting):
\begin{cor}[after Simon]\label{cor:sim+}
Assume (\ref{eq:assum}). Then almost* surely there exists a set $Q_{\omega}$ of zero logarithmic capacity (and, in particular, $\mes_\rho Q_\omega = 0$  for any gauge function $\rho$ satisfying (\ref{eq:condrho})) such that  all the spectral measures of $H_\omega$ are almost* surely supported on the union $Q_{\omega} \cup \{\gamma_W = 0\}$.
\end{cor}
Indeed, let $F = \left\{ \binom{v}{0} \, : \, v \in \mathbb R^n\right\}$; by Schnol's lemma (see e.g.\ \cite{Han}) all the spectral measures of $H_\omega$ are supported on the set of energies $E$ such that
\[\limsup_{n \to \infty}   \frac1n \log s_W(\Phi_{n, \omega}(z) \pi_F^*)  = 0~, \]
which, by Theorem~\ref{thm:limsup}, is contained  in the union of $\{\gamma_W = 0\}$ and a set of zero logarithmic capacity.

\medskip 
The result of Simon \cite{Simon} (the $W=1$ case of Corollary~\ref{cor:sim+}) strengthens and generalises the following  results. Jitomirskaya and Last \cite{JL} and Jitomirskaya \cite{Ji} showed for specific classes of almost periodic operators with positive Lyapunov exponent that the spectral measures are supported on a set of zero Hausdorff dimension. In an unpublished thesis, Landrigan \cite{Lan} showed that this set can be chosen to have zero $\rho$-Hausdorff measure for  $\rho(t) = \log^{-1-\epsilon}(e/t)$ (satisfying (\ref{eq:condrho})). Landrigan and Powell \cite{LP} showed that the same is true for $\rho(t) = \log^{-1}(e/t) \log^{-1-\epsilon} \log (e^e /t)$ (still satisfying (\ref{eq:condrho})); their work applies to general (not necessarily ergodic) operators. All of these results pertain to $W = 1$.

\begin{qn} For which gauge functions $\rho(t)$ can one find an  operator $H_\omega$ as in (\ref{eq:assum}) such that $\gamma_W > 0$ on $S$ and yet almost surely the spectral measures are not supported on any set $Q$ with $\mes_\rho Q = 0$? In particular, what can be said in the case  $\rho(t) = \log^{-1}(e/t)$?\end{qn}

In the unbounded, non-ergodic setting, $W=1$, Landrigan and Powell \cite[Theorem 2.4 (ii)]{LP} constructed operators for which any $Q$ supporting the spectral measure satisfies $\mes_\rho Q  > 0$ when $\rho(t)$ is arbitrarily close to $\log^{-1}(e/t)$ (e.g.\ $\rho(t) = \log^{-1+\epsilon}(e/t)$).  

\paragraph{Acknowledgement} We are grateful to Anton Gorodetski and to Barry Simon for helpful correspondence.

\section{Proof of Theorem~\ref{thm:liminf}}\label{s:liminf}

We start from the case when $\mathbf F$ is bounded, so that $K_E = \sup_n \|T_n(E)\| < \infty$.

Assume that $\liminf\limits_{n \to \infty} p_n s_W(\Phi_n(E)) > 0$ for all $E \in (a, b)$, $S \cap (a,b) \neq \varnothing$. For each $E \in (a,b)$, let $F^-_n(E) \subset \mathbb R^{2W}$ be the subspace spanned by the right singular vectors of $\Phi_n(E)$ (i.e.\ the eigenvectors of $\Phi_n(E)^*\Phi_n(E)$) corresponding to the $W$ smallest singular values. In case of degeneracy, we can still construct $F^-_n(E)$ to be Lagrangian and continuously dependent on $E$. By assumption, the singular values corresponding to the singular vectors in $F_n^-(E)$  lie in $(0, 1)$ for $n \geq n_0(E)$.

We claim that $F^-_n(E)$ converges, as $n \to \infty$, to a Lagrangian subspace $F^-(E)$. Indeed, let $\alpha_n$ be the largest principal angle between $F^-_n(E)$ and $F^-_{n+1}(E)$. 
For $n \geq n_0(E)$, there exist a unit vector $u_{n+1} \in F^-_{n+1}(E)$, and a pair of unit vectors $v_n^- \in F^-_n(E)$ and $v_n^+ \in F^-_n(E)^\perp$ so that
\[ u_{n+1} = \sin \alpha_n \, v_n^+ + \cos \alpha_n \,  v_n^-~. \]
Then, for 
\[ C_E = \left[ \inf_{n} p_n s_W(\Phi_n(E)) \right]^{-1}~,\]
we have (using for the third inequality below that the images of $v_n^\pm$ under $\Phi_n(E)$ are orthogonal):
\begin{equation}\label{eq:oseled} C_E p_{n+1} \geq \| \Phi_{n+1}(E) u_{n+1}\| \geq \frac{1}{K_E} \| \Phi_{n}(E) u_{n+1}\| 
\geq \frac{\sin \alpha_n}{K_E} \| \Phi_{n}(E)   \, v_n^+ \|  \geq \frac{\sin \alpha_n}{K_E \, C_E \, p_n}~, \end{equation}
whence by the Cauchy--Schwarz inequality 
\[ \sum_n \sin \alpha_n \leq K_E C_E^2 \sum p_n p_{n+1} \leq  K_E C_E^2 \sum p_n^2 < \infty~.\]
 Thus  $F^-_n(E)$ form a Cauchy sequence in the Lagrangian Grassmanian and converge to  a limit $F^-(E)$. Note that for each $u \notin F^-(E)$ we have $\liminf p_n \|\Phi_n(E) u\| >0$.

By Baire's continuity theorem (a corollorary of Baire's category theorem) $F^-(E)$ is continuous on a dense subset of $(a,b)$ (since it is a pointwise limit of a sequence of continuous functions). We now show that this is impossible.

Along with $H_\omega$, consider the family of self-adjoint operators $H_\omega^F$ associated with Lagrangian subspaces $F \subset \mathbb R^{2W}$: the operator $H_\omega^F$ acts on the space 
\[ \mathcal H_F = \left\{ \psi \in \ell_2(\mathbb Z_+ \to \mathbb C^W)~, \, \binom{\psi(1)}{\psi(0)} \in F \right\} \]
via
\[ (H_\omega^F \psi)(n) =  
\psi(n+1) + V_{n,\omega} \psi(n) + \psi(n-1)~, \quad n \geq 1~;\] 
$(H_\omega^F \psi)(0)$ is defined so that $H_\omega^F \psi \in \mathcal H_F$ (see Atkinson \cite{Atk} for details).
For almost* every $\omega$, all the operators $H_\omega^F$ are self-adjoint and have the same essential spectrum $S$. 

Fix a sequence $\tau_n \to +\infty$ such that $\sum_n \tau_n^2 p_n^2 < \infty$. By Schnol's lemma (see e.g.\ Han \cite{Han} for an argument applicable in the current situation), for almost every $E$ with respect to the spectral measure of $H_\omega^F$ (and in particular on a dense subset of $S \cap (a,b)$) there exists a formal solution $\psi_n^{F,E} \not\equiv 0$ of the equation $H_F \psi  = E \psi$ such that $\|\psi^{F,E}(n)\| \leq \tau_n^{-1} p_n^{-1}$. In particular,
\[ \|\Phi_n(E) \binom{\psi^{F,E}(1)}{\psi^{F,E}(0)} \| \leq 2 \tau_n^{-1} p_n^{-1} = o(p_n^{-1})~.\]
If $E$ lies in $(a,b)$,   this implies that $\binom{\psi^{F,E}(1)}{\psi^{F,E}(0)}  \in F^-(E)$, i.e.\ $F^-(E) \cap F \neq \{0\}$. 

For each $F$, this  holds on a dense subset of $S \cap (a,b)$. This implies that $F^-(E)$ is not continuous at any point of $S \cap (a,b)$: indeed, if $F^-(E)$ were continuous at $E_0$, the property would fail for $F = F(E_0)^\perp$. This contradiction shows that the set (\ref{eq:nonlyap-liminf}) is dense in $S$.  

This set is also $G_\delta$ in $S$, as its complement in $S$ is equal to the $F_\sigma$ set
\[ \bigcup_{m \geq 1} \bigcup_{n_0 \geq 1} \bigcap_{n \geq n_0} \left\{ E \in S  \, : \,  s_W(\Phi_n(E))  \geq \frac{1}{m p_n} \right\}~.\]
This completes the proof of item (b).

Now consider the general case, when $\mathbf F: \Omega \to \mathbb R$ is only assumed to satisfy (\ref{eq:logint}). We claim that the preceding arguments still apply to the sequence $p_n = e^{\epsilon n}$. Indeed, (\ref{eq:logint}) implies that
\[\sum_{n=1}^\infty \mathbb P \left\{ |\mathbf F| \geq e^{\epsilon n} \right\} \leq   \frac1\epsilon \int_0^\infty \mathbb P \left\{ \mathbf F \geq e^{r} \right\} dr < \infty~,\] 
whence almost surely 
\[ \limsup_{n \to \infty} \left[ |\mathbf F(T^n \omega)| e^{-\epsilon n} \right] \leq 1\]
by the first Borel--Cantelli lemma. Set
\[ K_E' =  \sup_{n} \|T_n(E)\| e^{-\epsilon n} \leq \sup_{n} \left[ |\mathbf F(T^n \omega)| e^{-\epsilon n}  \right] + |E| + 1  ~,\]
then instead of  (\ref{eq:oseled}) we have:
\[C_E p_{n+1}  \geq \frac{\sin \alpha_n}{K_E' \, C_E \, p_n e^{\epsilon n}}~, \]
whence 
\[ \sin \alpha_n \leq K_E' C_E^2 p_{n} p_{n+1} e^{\epsilon n} \leq K_E' C_E^2 e^{-\epsilon n}~. \]
From this point the proof proceeds as in case (b). \qed

\section{Proof of Theorem~\ref{thm:cs}}\label{s:cs}

The functions 
\[ \Gamma_{k, n}(z) = \frac{1}{n} \mathbb E \sum_{j=1}^k \log s_j(\Phi_{n,\omega}(z)) \]
converge pointwise to $\Gamma_k(z) = \sum_{j=1}^k \gamma_j(z)$. We claim that the sequence $(\Gamma_{k, n})_{n \geq 1}$ is precompact in $\mathcal D'$. Indeed, recall  (see e.g.\ H\"ormander \cite[Theorem 3.2.12]{Hor}) that a sequence $(u_n)$ of subharmonic functions is precompact  if for any $R > 0$
\begin{equation}\label{eq:upperbd} \limsup_{n\to\infty} \sup_{z \in B(0, R)} u_n(z) < \infty\end{equation}
and for some $r > 0$
\begin{equation}\label{eq:notminusinf}\liminf_{n \to \infty} \sup_{z \in B(0, r)} u_n(z) > -\infty~.\end{equation}
For our functions $\Gamma_{k,n}$, (\ref{eq:notminusinf}) follows from the pointwise inequality $\Gamma_{k,n} \geq 0$, while (\ref{eq:upperbd}) follows from the crude estimate
\[ \Gamma_{k, n}(z) \leq k\Gamma_{1,n}(z) \leq \frac k n \mathbb E \log \|\Phi_{n,\omega}(z)\| \leq  k \mathbb E \log \|T_{1,\omega}(z)\| \leq k (\mathbb{E} \log_+ \|V_{1, \omega}\| + \log_+ |z| + \log (4e))~.\]
Any limit point of $\Gamma_{k,n}$ in $\mathcal D'$ has to coincide with $\Gamma_k$, hence $\Gamma_{k,n} \to \Gamma_k$ in $\mathcal D'$.

Now consider the functions 
\[ \Gamma_{k, n,\omega}(z) = \frac{1}{n}  \sum_{j=1}^k \log s_j(\Phi_{n,\omega}(z))~.\]
As above, the property (\ref{eq:notminusinf}) still follows from $\Gamma_{k,n,\omega} \geq 0$. We claim that (\ref{eq:upperbd}) holds almost* surely. Indeed,  
\[\Gamma_{k,n,\omega}(z)\leq k ( \frac{1}{n} \sum_{m=1}^n \log_+ \|V_{m, \omega}\| + \log_+ |z| + \log (4e))~, \]
and 
\[ \frac{1}{n} \sum_{m=1}^n \log_+ \|V_{m, \omega}\|  \to \mathbb E \log_+ \|V_{1, \omega}\| \]
almost* surely by the Birkhoff ergodic theorem. Thus $\Gamma_{k,n, \omega} \to \Gamma_k$ in $\mathcal D'$ almost* surely. This concludes the proof of the first item  of the theorem.

The proof of the second item requires a more careful analysis of the transfer matrices. Denote 
\[ M_j(R) = \frac{1}{2\pi} \int_0^{2\pi}  \gamma_j(Re^{i\theta}) d\theta~. \]
We claim that
\begin{equation} \label{eq:loggrowth} \log R - o(1) \leq M_W(R) \leq M_1(R) \leq \log R + o(1)~, \quad R \to +\infty~.\end{equation}
Indeed, by the first item of the theorem and a general property of sequencess of subharmonic functions converging in $\mathcal D'$ topology (see \cite[Theorem 2.7.1.2]{Az}) we have almost surely (for any fixed $R>0$):
\[ M_j(R) = \lim_{n \to \infty} \frac{1}{2\pi} \int_0^{2\pi} \frac1n \log s_j(\Phi_{n,\omega}(Re^{i\theta})) d\theta~. \]
For the upper bound in (\ref{eq:loggrowth}), observe that
\[ \| T_{m,\omega}(R e^{i\theta}) \| \leq R + \| V_{m,\omega} \| + 1 = R \, ( 1 + \frac{\| V_{m,\omega}\|+1}{R})~,\]
whence for $R \geq 1$:
\[\begin{split} 
\| \Phi_{n, \omega} (Re^{i\theta})\| 
&\leq R^n \prod_{m=1}^n (1 + \frac{\|V_{m,\omega}\|+1}{|R|}) \\
&\leq R^n \exp(\frac{1}{R} \sum_{m=1}^n (\|V_{m,\omega}\| +1)\mathbbm 1_{\|V_{m,\omega}\| \leq R}
+ \sum_{m=1}^n \log \frac{3\|V_{m,\omega}\|}{R} \mathbbm 1_{\|V_{m,\omega}\| > R}) \\
&\leq R^n \exp(\frac{Cn}{R} + \frac{1}{R} \sum_{m=1}^n \|V_{m,\omega}\|\mathbbm 1_{\|V_{m,\omega}\| \leq R}
+ \sum_{m=1}^n \log \frac{\|V_{m,\omega}\|}{R} \mathbbm 1_{\|V_{m,\omega}\| > R})~.
\end{split}\]
As $n \to \infty$, we have almost surely (and independently of the argument $\theta$), for, say, $R \geq e$:
\[ \begin{split}
&\lim_{n \to \infty} \frac1{n} \sum_{m=1}^n  \|V_{m,\omega}\| \mathbbm{1}_{\|V_{m,\omega}\leq R}= \mathbb E \|V_{1,\omega}\| \mathbbm{1}_{\|V_{1,\omega}\|\leq R} \leq C +\frac{R}{\log R}  \mathbb{E} \log_+ \|V_{1,\omega}\| ~,\\
&\lim_{n \to \infty} \frac1n \sum_{m=1}^n \log_+  \|V_{m,n}\| \mathbbm 1_{\|V_{m,\omega}\| > R}= \mathbb E \log_+ \| V_{1, \omega}\| \mathbbm 1_{\|V_{1, \omega}\| \geq R}~,\end{split}\]
whence 
\[ M_1(R) \leq \sup_{\theta} \gamma_1(Re^{i\theta}) \leq \log R + O(1/R) + O(1/\log R) + o(1) = \log R + o(1)~, \quad R \to 
+ \infty~.\]
For the lower bound, we consider two cases. For $|\sin \theta| \geq 1/\sqrt R$, the cone 
\[ \left\{ \binom a b \in \mathbb C^{2W} \, : \, \|b \| \leq \frac 2 {\sqrt R} \|a\| \right\} \]
is preserved by any matrix 
\[\left( \begin{array}{cc} Re^{i\theta} -V & -\mathbbm 1 \\ \mathbbm 1 & 0 \end{array} \right)~, \]
and the norm of the first co\"ordinate is multiplied by a factor which is lower-bounded by 
\[ \begin{cases}
R - 2\sqrt R~, &\|V\|\leq \sqrt R\\
\frac12 R |\sin \theta|~, &\|V\| \geq \sqrt R~.
\end{cases}\]
Thus (almost surely)
\begin{equation}\label{eq:est-a}\begin{split}
\frac1n  \log s_W( \Phi_{n, \omega} (Re^{i\theta}) ) 
&\geq 
\frac 1 n \sum_{m=1}^n \left[ \log(R - 2\sqrt R) \mathbbm 1_{\|V_{m,\omega}\| \leq \sqrt R} + \log (\frac R 2 |\sin \theta|)  \mathbbm 1_{\|V_{m,\omega}\| > \sqrt R}\right] \\
&=\log R - O(\frac{1}{\sqrt R}) + \log \frac{|\sin \theta|}{2} o(1) \end{split}\end{equation}
as $R \to \infty$, with the $O(\cdot)$ and $o(\cdot)$ terms uniform in $\theta$. 

In the complementary case $|\sin \theta|< 1/\sqrt R$, the $W$-th singular value of $\Phi_{n,\omega}$ is estimated as follows. Let $G_z(m, n) = G_z[H_\omega](m,n)$ denote the $(m,n)$ $W\times W$ block of the block matrix representing the resolvent $(\tilde H_\omega - z)^{-1}$ of the full-line operator $\tilde H_\omega$ (acting on $\ell_2(\mathbb Z)$). Then 
\[ s_W(\Phi_{n, \omega} (Re^{i\theta}) ) 
\geq \left\| \binom{G_{Re^{i\theta}}(1, n+1)}{G_{Re^{i\theta}}(1, n)} \right\|^{-1} \left\| \binom{G_{Re^{i\theta}}(1, 1)}{G_{Re^{i\theta}}(1, 0)} \right\|~.\]
Now, the Combes-Thomas bound yields 
\[ \left\| \binom{G_{Re^{i\theta}}(1, n+1)}{G_{Re^{i\theta}}(1, n)} \right\|^{-1} \geq c  R |\sin \theta| \exp(c n \log(1 + R|\sin \theta|))~,\]
hence
\begin{equation}\label{eq:est-b} \mathbb E \log s_W(\Phi_{n, \omega} (Re^{i\theta}) )  
\geq  \log R + \log |\sin \theta| + c n \log(1 + R|\sin \theta|) - \operatorname{const}~.\end{equation}
Integrating (\ref{eq:est-a})--(\ref{eq:est-b})  over $\theta$ and letting $n \to \infty$, we obtain: $M_W(R) \geq \log R + o(1)$, as claimed.
This concludes the proof of (\ref{eq:loggrowth}). Now the second item of the theorem follows from this estimate and the next two standard lemmata, which we prove below for completeness. 
\qed

\begin{lemma}\label{l:logpot} Let $u: \mathbb C \to \mathbb R$ be a subharmonic function such that 
\[ M_u(R) \overset{\operatorname{def}}= \frac{1}{2\pi} \int_0^{2\pi} u(Re^{i\theta}) d\theta = \log R + o(1)~, \quad R \to +\infty~.\]
Then there exists a probability measure $\mu$ on $\mathbb C$ with
\[ \int \log_+ |z| d\mu(z) < \infty~, \quad u(z) = \int \log |z-z'| d\mu(z')~.\]
\end{lemma}
\begin{proof}
We can assume without loss of generality that $u(z)$ is harmonic in a small neighbourhood of the origin. Then (see Hayman and Kennedy \cite[Theorem 4.2]{HK}), $u(z)$ admits a representation of the form
\begin{equation}\label{eq:repr-hk} u(z) =  \int \log|1 - z/z'| d\mu(z') + A~,\end{equation}
where  $\mu = \frac{1}{2\pi} \Delta u$ is the Riesz measure of $u$. The measure $\mu$ has no atoms since $u > -\infty$.  Thus
\[ \begin{split}
M_u(R) &= A + \int d\mu(z') \frac{1}{2\pi} \int_0^{2\pi} \log |1 - Re^{i\theta}/z'| d\theta\\
&= A + \int d\mu(z') \log_+(R/|z'|) \geq A + \mu(B(0, r)) \log(R/r)~, \quad 0 < r < R~,
\end{split} \]
whence
\[ \mu(B(0,r)) \leq \inf_{R > r} \frac{M_u(R) - A}{\log(R/r)} \leq 1~, \quad \mu(\mathbb C) \leq 1~.\]
Next, 
\[ M_u(R) = A + \int d\mu(z') \log_+(R/|z'|) \leq A +  \mu(\mathbb C) \, \log R - \int_{B(0, R)} \log |z'| d\mu(z')~, \]
therefore 
\[  \mu(\mathbb C) = 1~, \quad  \int  \log  |z'| d\mu(z')   = \lim_{R \to \infty} \int_{B(0, R)} \log |z'| d\mu(z') \leq A~.\]
Now we can rewrite (\ref{eq:repr-hk}) as
\[ u(z) =  \int \log|z - z'| d\mu(z') + A'~, \quad A' = A - \int \log |z'| d\mu(z')~.\]
Then
\[ \log R + o(1) = M_u(R) = A' + \log R \,\, \mu(B(0, R)) + \int_{|z'| > R} \log|z'| d\mu(z') = A' + \log R + o(1)~,\]
thus $A' = 0$.
\end{proof}

\begin{lemma}[cf.\ \cite{CS, CS2}] Let $\mu$ be a probability measure on $\mathbb C$ such that 
\[\int \log_+ |z| d\mu(z) < \infty~, \quad \inf_{z \in \mathbb C}  \int \log |z-z'| d\mu(z') \geq 0~.\]
Then
\[\sup \left\{ \frac{\mu(B(z,r)) |\log r|}{1 + \log_+|z|}  \,\, : \,\,  z \in \mathbb C, \, 0 < r \leq 1/e   \right\} < \infty~.\]
\end{lemma}
\begin{proof} For any $z \in \mathbb C$,
\[ \begin{split}0 &\leq \int \log |z-z'| d\mu(z') = \int_{B(z, 1)} \log |z-z'| d\mu(z') + \log_+ |z|  + \int_{|z'-z| > |z|} \log_+ (2|z'|) d\mu(z') \\
&\leq- \int_{B(z, 1)} \log_- |z-z'| d\mu(z')  + (C_\mu + \log_+|z|)~,
\end{split}\]
and consequently for, say, $0 < r \leq 1/e$
\[ \frac{\mu(B(z, r))}{|\log r|} \leq \int_{B(z, 1)} \log_- |z-z'| d\mu(z')  \leq (C_\mu + \log_+|z|) \leq (C_\mu+1) (1 + \log_+ z)~.\qedhere\]
\end{proof}

\section{Proof of Theorem~\ref{thm:rho}}\label{s:rho}

By Theorem~\ref{thm:outsidesp}, we only need to consider $E \in \mathbb R$.

Let $\epsilon > 0$ and $R > 0$.  Let $n_i = \left\lfloor \exp(\tau i) \right\rfloor$, where $\tau$ is sufficiently small to ensure that 
\[ 4 \tau \sup_{E \in [-R, R]} (\gamma_1(E)+ \epsilon) < \epsilon~.\]
 We claim that (almost* surely) any $E \in [-R,R]$  such that 
\[ \liminf_{n \to \infty}  \frac{1}{n} \log  s_W(\Phi_{n,\omega}(E))  \leq \gamma_W(E)  - \epsilon\]
also satisfies
\begin{equation}\label{eq:also}\liminf_{i \to \infty}   \frac{1}{n_i} \log  s_W(\Phi_{n_i,\omega}(E)) \leq \gamma_W(E) - \frac{\epsilon}{2}~. \end{equation}
Indeed, the large deviation estimate (\ref{eq:largedev}) and an interpolation argument such as in \cite[Lemma 2.3]{GS} implies that almost* surely
\begin{align}
&\limsup_{n \to \infty} \frac{1}{n} \sum_{j=1}^k \log s_j (\Phi_{n,\omega}(E)) \leq \sum_{j=1}^k \gamma_j(E) + \frac{\epsilon}{10} \quad (1 \leq k  \leq W)~; \\
& \limsup_{i\to \infty} \frac{1}{n_{i+1}-n_i} \max_{1 \leq m < n_{i+1} - n_i} \log  \| T_{n_i+m, \omega}(E) \cdots T_{n_i,\omega}(E)\| \leq \gamma_1(E) + \epsilon~; 
\end{align}
then for $i$ large enough and $1 \leq m < n_{i+1} - n_i$
\[\begin{split}
&   \frac{1}{n_i + m} \log  s_W(\Phi_{n_i + m,\omega}(E)) -\gamma_W(E)    \\
&\quad \geq  \frac{1}{n_i} \log  s_W(\Phi_{n_i,\omega}(E)) -\gamma_W(E)   -
 \left| \frac{1}{n_i+m} \log  s_W(\Phi_{n_i+m,\omega}(E)) - \frac{1}{n_i} \log  s_W(\Phi_{n_i,\omega}(E)) \right| \\
&\quad \geq   \frac{1}{n_i} \log  s_W(\Phi_{n_i,\omega}(E)) -\gamma_W(E)  - 2 \, \frac{n_{i+1}-n_i}{n_i} (\gamma_1(E) + \epsilon)  \\
&\quad \geq \frac{1}{n_i} \log  s_W(\Phi_{n_i,\omega}(E)) -\gamma_W(E)  - \frac{\epsilon}{2}~,
\end{split}\]
 as claimed.

Further, if $E$ satisfies (\ref{eq:also}), then by Corollary~\ref{thm:cs-unif} 
\begin{equation}\label{eq:also2}\limsup_{i \to \infty}   \frac{1}{n_i} \sum_{j=1}^{W-1} \log  s_j(\Phi_{n_i,\omega}(E)) \leq \sum_{j=1}^{W-1} \gamma_j(E) - \frac{\epsilon}{2}~. \end{equation}

Now we construct a covering of the set of $E \in [-R,R]$ satisfying (\ref{eq:also2}). First cover $[-R, R]$ by a finite union of intervals $I_\alpha$ such that on each of them
\[ \max_j (\sup_{E \in I_\alpha} \gamma_j(E) - \inf_{E \in I_\alpha} \gamma_j(E) )\leq \frac{\epsilon}{4 W}~. \]
Then on each $I = [a_\alpha, b_\alpha]$ we have for any $E\in I_\alpha$ satisfying (\ref{eq:also2}):
\begin{equation}\label{eq:also3}\limsup_{i \to \infty}   \frac{1}{n_i} \sum_{j=1}^{W-1} \log  s_j(\Phi_{n_i,\omega}(E)) \leq \sum_{j=1}^{W-1} \gamma_j(a_\alpha) - \frac{\epsilon}{4}~. \end{equation}
Observe that $\sum_{j=1}^{W-1} \log s_j(\cdot)$ is the logarithm of the operator norm of the $(W-1)$-st wedge power of the matrix $(\cdot)$, and that the relation (\ref{eq:also3}) remains valid if the operator norm is replaced with the maximum $||| (\cdot)^{\wedge (W-1)} |||$ of the matrix entries. Now,  the set of $E \in I_\alpha$  satisfying
\[ \frac{1}{n_i} \log ||| \Phi_{n_i,\omega}(E)^{\wedge(W-1)}||| \leq \sum_{j=1}^{W-1} \gamma_j(a_\alpha) - \frac{\epsilon}{4} \]
 is the union of $O(n)$ intervals $I_{\alpha,\beta}^{(i)}$. By the large deviation estimate (\ref{eq:largedev}), the length of each of these intervals is exponentially small in $n_i$ (almost surely for $i$ large enough). Thus (using (\ref{eq:condrho}))
\[ \sum_i \sum_{\alpha,\beta} \rho(\operatorname{mes} I_{\alpha, \beta}^{(i)})  \leq C \sum_i n_i \rho(e^{-c n_i}) < \infty~,\]
whence
\[ \mes_{\rho} \bigcap_{i=1}^\infty\bigcup_{j\geq i} \bigcup_{\alpha,\beta} I_{\alpha,\beta}^{(j)} = 0~.\tag*{\qed}\]

\section{Proof of Theorem~\ref{thm:limsup}}\label{s:limsup}
Consider the sequence of functions
\[ u_{n,\omega}(z) = \sup_{n' \geq n} \frac{1}{n} \sum_{j=1}^W \log s_j(\Phi_{n',\omega}(z)\pi_F^*)~, \quad n \geq 1~.\]
Since
\[ \sum_{j=1}^W \log s_j(\Phi_{n,\omega}(z)\pi_F^*) \leq \sum_{j=1}^W \log s_j(\Phi_{n,\omega}(z))~, \]
Theorem~\ref{thm:cs} implies that these functions are almost* surely finite. According to a theorem of Cartan \cite[Theorem 2.7.3.1]{Az}, for each $n$ there  exists a subharmonic function $u^*_{n,\omega}(z)$ such that $u_{n,\omega}= u^*_{n,\omega}$ outside a set of zero logarithmic capacity. Let 
\[ u^*_{\omega}(z) = \lim_{n \to \infty} u^*_{n,\omega}(z)~;\]
it is a limit of a non-increasing sequence of subharmonic functions, hence it is either subharmonic or identically $-\infty$. Also,
\[ \limsup_{n \to \infty} \frac{1}{n} \sum_{j=1}^W \log s_j(\Phi_{n,\omega}(z)\pi_F^*) = u^*_{\omega}(z) \]
outside a set of zero logarithmic capacity. 
We shall now prove that almost* surely
\begin{equation}\label{eq:limsup2}
\forall z \in \mathbb C \setminus \mathbb R \,\, \lim_{n \to \infty} u_{n,\omega}(z) = \Gamma_W(z)~,
\end{equation}
whence $u^*_{\omega}= \Gamma_W$ (outside a set of zero measure and hence everywhere on $\mathbb C$). 
This will imply that  
\[ \limsup_{n \to \infty} \frac{1}{n} \sum_{j=1}^W \log s_j(\Phi_{n,\omega}(z)\pi_F^*) = \Gamma_W(z) \]
outside a set of zero logarithmic capacity. This, combined  with the uniform Craig--Simon estimate (\ref{eq:upper}) for $k=W-1$, yields (\ref{eq:limsup1}).

To prove (\ref{eq:limsup2}), recall that by Theorem~\ref{thm:outsidesp}, almost* surely
\begin{equation}\label{eq:rec-unifhyp}\lim_{n \to \infty} \frac{1}{n} \sum_{j=1}^W \log s_j(\Phi_{n,\omega}(z)) = \Gamma_W(z)~, \quad z \in \mathbb C \setminus \mathbb R~.\end{equation}
For every $z \in \mathbb C \setminus \mathbb R$ there is a splitting $\mathbb C^{2W} = F^+(z) \oplus F^-(z)$ into a direct sum of two subspaces such that for every $v \in F^{-}(z)$ the norm $\| \Phi_n(z)v \|$ decays exponentially. Due to (\ref{eq:rec-unifhyp}), for any Lagrangian subspace $F$ such that $F \cap F^{-}(z) = \{0\}$ we have
\[ \lim_{n \to \infty} \frac{1}{n} \sum_{j=1}^W \log s_j(\Phi_{n,\omega}(z) \pi_F^*) = \Gamma_W(z)~,\]
which is what we need to prove. 
On the other hand, if $F \cap F^{-}(z)$ would  contain a non-zero vector,  then the self-adjoint operator $H_\omega^F$ as defined in Section~\ref{s:liminf} would have an eigenvalue at $z\in\mathbb C\setminus \mathbb R$, which is impossible. \qed

\end{document}